\documentclass[11pt,reqno]{amsart}

\usepackage[T1]{fontenc}
\usepackage{lmodern}
\usepackage{microtype}
\usepackage[margin=1.03in]{geometry}
\usepackage{amsmath,amssymb,amsfonts,mathtools}
\usepackage{amsthm}
\usepackage{booktabs}
\usepackage{enumitem}
\usepackage{xcolor}
\usepackage{url}
\usepackage[colorlinks=true,
  linkcolor=blue!50!black,
  citecolor=blue!50!black,
  urlcolor=blue!50!black]{hyperref}
\usepackage[capitalize,noabbrev]{cleveref}

\allowdisplaybreaks
\numberwithin{equation}{section}
\setlist[enumerate]{leftmargin=2.2em,itemsep=0.25em,topsep=0.4em}
\setlist[itemize]{leftmargin=2em,itemsep=0.25em,topsep=0.4em}

\theoremstyle{plain}
\newtheorem{theorem}{Theorem}[section]
\newtheorem{proposition}[theorem]{Proposition}
\newtheorem{lemma}[theorem]{Lemma}
\newtheorem{corollary}[theorem]{Corollary}

\theoremstyle{definition}

\theoremstyle{remark}

\newcommand{\F}{\mathbf F}
\newcommand{\Spec}{\operatorname{Spec}}
\newcommand{\Frac}{\operatorname{Frac}}
\newcommand{\Span}{\operatorname{Span}}
\newcommand{\fm}{\mathfrak m}
\newcommand{\cS}{\mathcal S}

\hypersetup{
  pdftitle={Polynomial extensions do not preserve the strong finite type property},
  pdfauthor={}
}

\title[Polynomial extensions do not preserve the strong finite type property]
  {Polynomial extensions do not preserve \\ the strong finite type property}

\author{Viet-Hoang Tran}
\address{Department of Mathematics, National University of Singapore, Singapore 119076}
\email{hoang.tranviet@u.nus.edu}
\urladdr{https://vh-tran.github.io/}

\author{Phan Thanh Toan}
\address{Analytical and Algebraic Methods in Optimization Research Group, Faculty of Mathematics and Statistics, Ton Duc Thang University, Ho Chi Minh City, Vietnam}
\email{phanthanhtoan@tdtu.edu.vn}

\author{Thieu N. Vo}
\address{Department of Computer Science, University of Bath, United Kingdom}
\email{ntv22@bath.ac.uk}

\author{Tan M. Nguyen}
\address{Department of Mathematics, National University of Singapore, Singapore 119076}
\email{tanmn@nus.edu.sg}
\urladdr{https://tanmnguyen89.github.io/}

\date{}

\subjclass[2020]{Primary 13E05; Secondary 13F20, 13A15, 13F05}
\keywords{strong finite type, polynomial ring}

\begin{document}
\raggedbottom

\begin{abstract}
Arnold introduced the strong finite type (SFT) property in 1973 while studying
the dimension of power series rings.  For several classes of rings, polynomial
extension is known to preserve the SFT property, but the general question
remained open.  We answer it negatively by constructing an SFT ring $R$ such
that $R[X]$ is not SFT.
\end{abstract}

\maketitle

\section{Introduction}

Throughout the paper, rings are commutative and have identity.  Following
Arnold \cite{Arnold1973Krull}, an ideal $I$ of a ring $R$ is of
\emph{strong finite type}, or \emph{SFT}, if there are a finitely generated
ideal $B\subseteq I$ and an integer $N\geq1$ such that
\begin{equation}\label{eq:sft-definition}
  a^N\in B\qquad\text{for every }a\in I.
\end{equation}
The ring $R$ is an SFT ring if every ideal of $R$ is SFT.  Noetherian rings
are SFT, by taking $B=I$ and $N=1$.  The point of the definition is that
$I$ itself may be infinitely generated: after passing to $R/B$, all
elements coming from $I$ are nilpotent with one common bound on their
nilpotency indices.  This is an elementwise condition; it does not say that
$I^N\subseteq B$.

\subsection{Why SFT rings arose}

Arnold introduced the SFT condition in his work on prime ideals and Krull
dimension in formal power-series rings
\cite{Arnold1973Krull,Arnold1973Prufer}.  The finiteness issue behind the
definition is particularly visible for formal series.  Given an ideal
$I\subseteq R$, set
\[
 I[[X]]:=\left\{\sum_{j\geq0}a_jX^j:a_j\in I\right\}.
\]
This coefficientwise ideal can be larger than the extended ideal
$IR[[X]]$: a series with coefficients in $I$ need not have all its
coefficients in one finitely generated subideal of $I$.  Arnold showed that
the SFT condition is equivalent to the radical equality
\[
  P[[X]]=\sqrt{PR[[X]]}
  \qquad\text{for every }P\in\Spec R,
\]
and proved in particular that
\begin{equation}\label{eq:arnold}
  R\text{ not SFT}\quad\Longrightarrow\quad
  \dim R[[X]]=\infty.
\end{equation}
See \cite[Theorem~1]{Arnold1973Krull}; a later account is given by
Roitman \cite{Roitman2015}.  Thus finite dimensionality of $R[[X]]$ forces
$R$ to be SFT.  The condition arose from the formal power-series problem
rather than simply as an abstract weakening of Noetherianity.

Arnold's result led to two questions: if $R$ is SFT, must
$R[[X]]$ be SFT, and, when $R$ is also finite dimensional, must $R[[X]]$
remain finite dimensional?  Coykendall answered both questions negatively
by constructing a finite-dimensional SFT domain whose one-variable formal
power-series ring is neither SFT nor finite dimensional
\cite{Coykendall2002}.  Positive results survive for important classes.  For
example, formal power-series rings over SFT Pr\"ufer domains are SFT
\cite{KangPark2007}.  Related work has treated zero-dimensional SFT rings
\cite{CondoCoykendallDobbs1996} and the localization of $R[[X]]$ at the
power series having unit content, denoted $R((X))$ in
\cite{ElaoudKang2006}.  These results reflect the central role of formal
power-series extensions in the early SFT literature.

\begin{table}[!t]
\centering
\caption{The main SFT extension questions.}
\label{tab:sft-extension-status}
\footnotesize
\setlength{\tabcolsep}{5pt}
\renewcommand{\arraystretch}{1.15}
\begin{tabular}{@{}p{0.40\textwidth}p{0.55\textwidth}@{}}
\toprule
Question & Current status \\
\midrule
If $\dim R[[X]]<\infty$, must $R$ be SFT?
  & Yes, by Arnold \cite[Theorem~1]{Arnold1973Krull}. \\ \addlinespace[0.25em]
If $R$ is finite dimensional and SFT, must $\dim R[[X]]<\infty$?
  & No, even for a one-dimensional SFT domain \cite{Coykendall2002}. \\ \addlinespace[0.25em]
If $R$ is SFT, must $R[[X]]$ be SFT?
  & No in general \cite{Coykendall2002}; yes for SFT Pr\"ufer domains
    \cite{KangPark2007}, without a finite-dimensional hypothesis. \\ \addlinespace[0.25em]
If $I$ is an SFT ideal of $R$, must $IT$ be SFT for every ring extension
$R\subseteq T$?
  & Yes \cite[Theorem~4.13]{CoykendallDutta2024}; in particular, $IR[X]$ is
    SFT. \\ \addlinespace[0.25em]
If $R$ is SFT, must $R[X]$ be SFT?
  & Yes under Park's additional hypotheses \cite{Park2019}; the general
    question was still listed as open in 2024 \cite{CoykendallDutta2024}.
    \textbf{It is false in general by \cref{thm:main}.} \\
\bottomrule
\end{tabular}
\end{table}

\subsection{The polynomial-extension question}

The corresponding problem for $R[X]$ followed a different course.  Park
proved that the SFT property passes to polynomial extensions in several
substantial cases: zero-dimensional SFT rings, SFT Pr\"ufer domains, and SFT
domains $R$ such that $R/P$ is integrally closed for every
$P\in\Spec R$ \cite{Park2019}.  These results did not settle the general
question.

Coykendall and Dutta later proved that if $I$ is an SFT ideal of $R$, then
its extension to any ring extension of $R$ is again SFT.  In particular,
$IR[X]$ is SFT \cite[Theorem~4.13]{CoykendallDutta2024}.  In the same work
they still recorded the ring-level implication
\[
  R\text{ SFT}\quad\Longrightarrow\quad R[X]\text{ SFT}
\]
as open.  There is no contradiction between these two statements: a prime
ideal of $R[X]$ need not be extended from its contraction to $R$.  This is
exactly where our counterexample occurs.

The principal extension questions and their present status are collected in
\cref{tab:sft-extension-status}.  The last row isolates the question answered
in this paper.

Our main result is the following.

\begin{theorem}\label{thm:main}
There is a one-dimensional local SFT domain $R$ such that $R[X]$ is not
SFT.  More precisely, $R[X]$ has a non-SFT prime ideal that is an upper to
zero.
\end{theorem}

Recall that, for a domain $D$, a nonzero prime ideal of $D[X]$ whose
contraction to $D$ is $(0)$ is called an \emph{upper to zero}; this is the
standard terminology used in the polynomial-ring literature
\cite[p.~168]{ChangFontana2009}.  The prime in
\cref{thm:main} is the kernel of evaluation at an element of
$\Frac(R)$.  It is therefore invisible to the theorem about extended
ideals.

\subsection{Organization}

The next two sections give the construction and prove \cref{thm:main},
including the localization and coefficient comparisons needed for the
non-SFT prime.  Only after that proof do we record extensions of the example.
The final section gives consequences, compares the construction with earlier
positive results, and records the precise scope of the examples.

\section{The coefficient ring in characteristic two}
\label{sec:char2}

Let $u$ be transcendental over $\F_2$ and put $K=\F_2(u)$.  Choose an
enumeration
\[
  K^\times=\{\xi_1,\xi_2,\ldots\},\qquad \xi_1=u.
\]
For $n\geq0$, let $V_n$ be the $\F_2$-span of the monomials
\begin{equation}\label{eq:Vn-char2}
  \prod_{\ell=1}^r \xi_{i_\ell}^{2^{e_\ell}},
  \qquad r\geq0,\quad e_\ell\geq0,\quad
  \sum_{\ell=1}^r i_\ell\leq n,
\end{equation}
where repeated factors are allowed and the empty product is $1$.  The index
$i$ measures the cost of using $\xi_i$, while taking a Frobenius power does
not increase that cost.  This is why squaring will preserve each $V_n$.

\begin{proposition}\label{prop:V-properties}
The spaces $V_n$ have the following properties:
\begin{enumerate}[label=\textup{(\alph*)}]
\item $V_0=\F_2$, $V_n\subseteq V_{n+1}$, and
      $\bigcup_{n\geq0}V_n=K$;
\item $V_nV_m\subseteq V_{n+m}$ for all $m,n\geq0$;
\item $v^2\in V_n$ for every $v\in V_n$;
\item $\F_2[u]\nsubseteq V_E$ for every $E\geq0$.
\end{enumerate}
\end{proposition}

\begin{proof}
A monomial for which the sum of the indices is zero is the empty product, so
$V_0=\F_2$.  Increasing the allowed sum of indices gives
$V_n\subseteq V_{n+1}$.  Every nonzero element of $K$ occurs as some
$\xi_i$ and hence belongs to $V_i$; this proves exhaustivity.
Multiplication concatenates the factors in \eqref{eq:Vn-char2} and adds the
sums of their indices, proving (b).

Write $v=\sum_\alpha a_\alpha m_\alpha$, where $a_\alpha\in\F_2$ and the
$m_\alpha$ are monomials from \eqref{eq:Vn-char2}.  Then
\[
  v^2=\sum_\alpha a_\alpha^2m_\alpha^2.
\]
Squaring replaces each exponent $2^{e_\ell}$ by $2^{e_\ell+1}$ and leaves
the sum of the indices unchanged.  Thus $v^2\in V_n$.

It remains to prove (d).  Fix $E\geq0$ and, for $L\geq1$, set
\[
  F_L:=K^{2^L}=\F_2(u^{2^L}).
\]
Writing $t=u^{2^L}$, the polynomial
$T^{2^L}-t\in\F_2[t][T]$ is Eisenstein at the prime
$t\in\F_2[t]$.  Gauss's lemma therefore makes it irreducible over
$\F_2(t)=F_L$.  Consequently
\begin{equation}\label{eq:basis-char2}
  1,u,u^2,\ldots,u^{2^L-1}
\end{equation}
is an $F_L$-basis of $K$.

Consider a monomial occurring in $V_E$.  A factor $\xi_i^{2^e}$ with
$e\geq L$ lies in $F_L$.  After these factors are absorbed into the scalar
field, the remaining product has at most $E$ factors, each chosen from
\[
  \{\xi_i^{2^e}:1\leq i\leq E,\ 0\leq e<L\},
\]
a set of at most $EL$ elements.  If $F_LV_E$ denotes the $F_L$-linear span
of $V_E$, then
\begin{equation}\label{eq:dimension-char2}
  \dim_{F_L}(F_LV_E)\leq \sum_{r=0}^E(EL)^r.
\end{equation}
For $E=0$, the right side is understood as $1$, which also follows directly
from $V_0=\F_2$.  For fixed $E$, the right side grows polynomially in $L$,
whereas $2^L$ grows exponentially.  For sufficiently large $L$, therefore,
the space $F_LV_E$ cannot contain all the $2^L$ basis elements in
\eqref{eq:basis-char2}.  At least one power of $u$ is outside $V_E$, which
proves (d).
\end{proof}

The last paragraph is the only growth argument in the construction.  A
fixed level $V_E$ uses at most $E$ low Frobenius factors, so its dimension
after extending scalars to $F_L$ grows at most polynomially with $L$.  The
field $K$, by contrast, displays $2^L$ independent consecutive powers of
$u$.  This polynomial-versus-exponential comparison prevents any fixed
filtration level from containing $\F_2[u]$.

Since $\xi_1=u$, multiplication and squaring give
\begin{equation}\label{eq:u-bounds-char2}
  u\in V_1,\qquad u^j\in V_j\ (j\geq0),\qquad
  u^{2^n}\in V_1\ (n\geq0).
\end{equation}

Let $c$ be an indeterminate over $K$, and define
\[
  A:=\bigoplus_{n\geq0}c^nV_n\subseteq K[c],\qquad
  \fm:=A_+=\bigoplus_{n\geq1}c^nV_n,
\]
\begin{equation}\label{eq:R-char2}
  R:=A_{\fm},\qquad M:=\fm R.
\end{equation}
Property (b) makes $A$ a graded subring of the domain $K[c]$.  Taking the
constant coefficient gives $A/\fm\cong\F_2$, so $\fm$ is maximal and $R$
is a local domain with maximal ideal $M$.  Notice also that $c^n\in A$ for
all $n$, since $1\in V_n$.

Localization introduces denominators, so the coefficient restrictions in
the definition of $A$ are not automatic for elements of $R$.  The next
lemma supplies the required control.

\begin{lemma}\label{lem:formal-expansion}
Every $r\in R$ has a unique formal $c$-adic expansion
\begin{equation}\label{eq:formal-expansion}
  r=\sum_{n\geq0}c^nr_n\in K[[c]],\qquad r_n\in V_n.
\end{equation}
\end{lemma}

\begin{proof}
The set
\[
  \cS:=\left\{\sum_{n\geq0}c^na_n\in K[[c]]:a_n\in V_n\right\}
\]
is a subring of $K[[c]]$: the coefficient of $c^n$ in a product belongs to
\[
  \sum_{i=0}^nV_iV_{n-i}\subseteq V_n.
\]
Let $s\in A\setminus\fm$.  Its constant coefficient $s_0$ is nonzero, and
after multiplying by $s_0^{-1}\in\F_2$ we may write
\[
  s_0^{-1}s=1+\sum_{i=1}^Nc^is_i,\qquad s_i\in V_i.
\]
The inverse in $K[[c]]$ is $\sum_{n\geq0}c^nt_n$, where $t_0=1$ and
\[
  t_n=-\sum_{i=1}^{\min\{n,N\}}s_it_{n-i}.
\]
Induction and $V_iV_{n-i}\subseteq V_n$ show that $t_n\in V_n$.  Thus
$s^{-1}\in\cS$, and every fraction defining an element of $R$ lies in
$\cS$.  The canonical maps
\[
  R\hookrightarrow K[c]_{(c)}\hookrightarrow K[[c]]
\]
are injective, so the expansion is independent of the chosen fraction and
is unique.
\end{proof}

We now prove the two estimates that make every ideal of $R$ SFT.

\begin{lemma}\label{lem:two-estimates}
The following statements hold.
\begin{enumerate}[label=\textup{(\alph*)}]
\item If $z\in M$, then $z^2\in cR$.
\item If $0\neq a\in M$, then $c^D\in aR$ for some $D\geq1$.
\end{enumerate}
\end{lemma}

\begin{proof}
For (a), first take $b=\sum_{n=1}^Nc^nb_n\in\fm$.  Then
\begin{equation}\label{eq:square-control}
  \frac{b^2}{c}=\sum_{n=1}^Nc^{2n-1}b_n^2\in A,
\end{equation}
because $b_n^2\in V_n\subseteq V_{2n-1}$.  Writing an arbitrary element of
$M$ as $b/s$ with $s\notin\fm$ gives $(b/s)^2\in cR$.

For (b), write $a=b/t$ with $b\in\fm$ and
$t\in A\setminus\fm$.  Since $t$ is a unit of $R$, it is enough to work
with
\[
  b=\sum_{j=\nu}^Nc^jb_j,\qquad
  \nu\geq1,\quad b_\nu\neq0,\quad b_j\in V_j.
\]
Choose $s\geq0$ with $b_\nu^{-1}\in V_s$.  There are only finitely many
nonzero coefficients above $b_\nu$, so we may choose $r\geq0$ such that
\begin{equation}\label{eq:r-choice-char2}
  (j-\nu)2^r\geq j+s
  \quad\text{whenever }j>\nu\text{ and }b_j\neq0.
\end{equation}
Repeated squaring gives $b_\nu^{-2^r}\in V_s$, and hence
$\lambda:=c^sb_\nu^{-2^r}\in A$.  Moreover,
\begin{equation}\label{eq:ratio-char2}
  \left(\frac{b_j}{b_\nu}\right)^{2^r}
  \in V_jV_s\subseteq V_{j+s}
  \subseteq V_{(j-\nu)2^r}.
\end{equation}
The Frobenius homomorphism now gives
\begin{equation}\label{eq:divisibility-char2}
  \lambda b^{2^r}=c^{\nu2^r+s}
  \left(1+\sum_{j>\nu}c^{(j-\nu)2^r}
  \left(\frac{b_j}{b_\nu}\right)^{2^r}\right)
  =c^{\nu2^r+s}(1+w),
\end{equation}
where \eqref{eq:ratio-char2} shows that $w\in\fm$.  The element $1+w$ is
a unit in $R$.  Thus, with $D=\nu2^r+s$,
\[
  c^D=(1+w)^{-1}\lambda b^{2^r}\in bR=aR.
\]
\end{proof}

\begin{theorem}\label{thm:R-char2}
The ring $R$ is a one-dimensional local SFT domain, and
\[
  \Spec R=\{(0),M\}.
\]
\end{theorem}

\begin{proof}
Let $I$ be a nonzero proper ideal of $R$ and choose $0\neq a\in I$.
Because $R$ is local, $I\subseteq M$.  By
\cref{lem:two-estimates}(b), there is $D\geq1$ with $c^D\in aR$.  For
every $z\in I$,
\[
  z^{2D}=(z^2)^D\in c^DR\subseteq aR.
\]
The principal ideal $(a)$ is contained in $I$, so $I$ is SFT with the
finitely generated subideal $(a)$ and exponent $2D$.  The zero and unit
ideals are SFT directly, and hence every ideal of $R$ is SFT.

The ring $R$ is a domain, so $(0)$ is prime.  If $P$ is a nonzero prime,
choose $0\neq a\in P$.  Part (b) gives $c^D\in aR\subseteq P$, hence
$c\in P$.  For $z\in M$, part (a) gives $z^2\in cR\subseteq P$, so
$z\in P$.  Thus $M\subseteq P$, and maximality gives $P=M$.  Since
$0\neq c\in M$, these two primes are distinct and $\dim R=1$.
\end{proof}

\section{A non-SFT upper to zero in \texorpdfstring{$R[X]$}{R[X]}}
\label{sec:bad-prime}

We first identify the fraction field of $R$.  If $q\in K$, choose $n$ with
$q\in V_n$.  Then $c^nq,c^n\in A$, so
\[
  q=\frac{c^nq}{c^n}\in\Frac(R).
\]
Thus $K\subseteq\Frac(R)$; since $c\in R$, this gives
$K(c)\subseteq\Frac(R)$.  The reverse inclusion follows from
$R\subseteq K(c)$.  Therefore
\begin{equation}\label{eq:fraction-field}
  \Frac(R)=K(c).
\end{equation}
Consider evaluation at $u$:
\begin{equation}\label{eq:Q}
  Q:=\ker\bigl(R[X]\longrightarrow K(c),\ X\longmapsto u\bigr).
\end{equation}
The quotient $R[X]/Q\cong R[u]$ is a domain, so $Q$ is prime, and the map
is injective on $R$, so $Q\cap R=(0)$.  The ideal is nonzero because
$c(X-u)=cX-cu\in R[X]$: here $c\in A$ and $cu\in cV_1\subseteq A$.
Thus $Q$ is an upper to zero.

Division by $X-u$ shifts the coefficient bound by at most the degree.

\begin{lemma}\label{lem:division-char2}
Let $d\geq1$, and let $h\in R[X]$ have degree at most $d$ and satisfy
$h(u)=0$ in $K(c)$.  Write $h=(X-u)\widetilde h$ in $K(c)[X]$.  Then
$\widetilde h\in K[c]_{(c)}[X]$, and, for every $n\geq0$, the coefficient
of $c^n$ in each $X$-coefficient of $\widetilde h$ belongs to
$V_{n+d-1}$.
\end{lemma}

\begin{proof}
Write $h=\sum_{j=0}^dh_jX^j$.  Since $h(u)=0$,
\begin{equation}\label{eq:division-formula}
 h(X)=h(X)-h(u)=\sum_{j=1}^dh_j(X^j-u^j)
 =(X-u)\sum_{j=1}^dh_j\sum_{t=0}^{j-1}X^{j-1-t}u^t.
\end{equation}
Since $h_j\in R\subseteq K[c]_{(c)}$ and $u\in K$, the displayed quotient
belongs to $K[c]_{(c)}[X]$.
By \cref{lem:formal-expansion}, the coefficient of $c^n$ in $h_j$ lies in
$V_n$.  Equation \eqref{eq:u-bounds-char2} gives $u^t\in V_t$.  Hence the
coefficient of $c^n$ in every coefficient of the last polynomial in
\eqref{eq:division-formula} lies in
\[
  \sum_{t=0}^{d-1}V_nV_t\subseteq V_{n+d-1}.
\]
\end{proof}

\begin{theorem}\label{thm:Q-not-sft}
The prime $Q$ is not SFT.  Consequently $R[X]$ is not SFT.
\end{theorem}

\begin{proof}
For $n\geq0$, \eqref{eq:u-bounds-char2} and the Frobenius homomorphism give
\begin{equation}\label{eq:gn-char2}
  g_n:=c(X-u)^{2^n}=cX^{2^n}+cu^{2^n}\in R[X].
\end{equation}
Indeed, $u^{2^n}\in V_1$, so $cu^{2^n}\in cV_1\subseteq A$.
Evaluation at $u$ vanishes, and hence $g_n\in Q$.

Suppose that $Q$ were SFT.  Then there would be a finitely generated ideal
$J=(h_1,\ldots,h_t)\subseteq Q$ and an integer $e\geq1$ such that
\begin{equation}\label{eq:sft-assumption}
  f^e\in J\qquad\text{for every }f\in Q.
\end{equation}
The ideal $J$ is nonzero, since otherwise $g_0^e=0$ in the domain $R[X]$.
Discarding zero generators if necessary, put
\[
  d:=\max_i\deg h_i.
\]
Because $Q\cap R=(0)$, each $h_i$ has positive degree and $d\geq1$.

Choose $r\geq0$ so that $q:=2^r\geq e$.  Then $f^q\in J$ for every
$f\in Q$, because $f^q=f^{q-e}f^e$.  For each $n$ there are
$a_{i,n}\in R[X]$ such that
\begin{equation}\label{eq:representation-char2}
  g_n^q=\sum_{i=1}^ta_{i,n}h_i.
\end{equation}
By \cref{lem:division-char2}, write
$h_i=(X-u)\widetilde h_i$.  Substituting \eqref{eq:gn-char2} into
\eqref{eq:representation-char2} and cancelling the nonzero polynomial
$X-u$ in the domain $K(c)[X]$ gives
\begin{equation}\label{eq:cancelled-char2}
  c^q(X-u)^{2^{n+r}-1}
    =\sum_{i=1}^ta_{i,n}\widetilde h_i.
\end{equation}

Formula \eqref{eq:division-formula} puts every coefficient of
$\widetilde h_i$ in $K[c]_{(c)}$; the coefficients of $a_{i,n}$ already
lie there because $R\subseteq K[c]_{(c)}$.  The injective map
\[
  K[c]_{(c)}[X]\longrightarrow K[[c]][X]
\]
therefore allows us to compare the coefficient of $c^q$ on both sides.
Write
\[
  a_{i,n}=\sum_\alpha A_{i,\alpha}(c)X^\alpha,
  \qquad
  \widetilde h_i=\sum_\beta H_{i,\beta}(c)X^\beta.
\]
The formal expansions from
\cref{lem:formal-expansion,lem:division-char2} give
\[
  [c^s]A_{i,\alpha}\in V_s,
  \qquad [c^\ell]H_{i,\beta}\in V_{\ell+d-1}\subseteq V_{\ell+d}.
\]
For a fixed $X$-degree $m$,
\begin{equation}\label{eq:convolution-char2}
 [c^qX^m](a_{i,n}\widetilde h_i)
 =\sum_{\alpha+\beta=m}\ \sum_{s+\ell=q}
   [c^s]A_{i,\alpha}\,[c^\ell]H_{i,\beta}.
\end{equation}
Both sums are finite: the factors are polynomials in $X$, and there are only
$q+1$ pairs $(s,\ell)$ of nonnegative integers with $s+\ell=q$.  Every
summand belongs to
\[
  V_sV_{\ell+d}\subseteq V_{q+d}.
\]
Since $V_{q+d}$ is a vector space, the same remains true after all sums and
cancellations on the right side of \eqref{eq:cancelled-char2}.

Since $u$ is constant with respect to $c$, the coefficient of $c^q$ on the
left side is exactly $(X-u)^{2^{n+r}-1}$.  Every coefficient of this
polynomial therefore belongs to the one fixed space $V_{q+d}$, independent
of $n$.

Set $N=n+r$.  In characteristic $2$,
\[
  (X-u)^{2^N}=X^{2^N}-u^{2^N}.
\]
The difference-of-powers identity yields
\begin{equation}\label{eq:geometric-char2}
  (X-u)^{2^N-1}
   =\frac{X^{2^N}-u^{2^N}}{X-u}
   =\sum_{j=0}^{2^N-1}u^jX^{2^N-1-j}.
\end{equation}
It follows that
\[
  1,u,\ldots,u^{2^{n+r}-1}\in V_{q+d}
\]
for every $n$.  Letting $n$ grow gives
$\F_2[u]\subseteq V_{q+d}$, contradicting
\cref{prop:V-properties}(d).  Hence $Q$ is not SFT.
\end{proof}

Combining \cref{thm:R-char2,thm:Q-not-sft} proves
\cref{thm:main}.

\section{Extension to prime characteristic}
\label{sec:general-p}

We record after the main proof that the same construction works for every
prime characteristic.  Only the Frobenius exponent and the dimension count
change.

\begin{theorem}\label{thm:general-p}
For every prime $p$, there is a one-dimensional local SFT domain $R_p$ of
characteristic $p$ such that $R_p[X]$ is not SFT.  The non-SFT prime may be
chosen to be an upper to zero.
\end{theorem}

\begin{proof}
Let $u$ be transcendental over $\F_p$, put $K=\F_p(u)$, and enumerate
\[
  K^\times=\{\xi_1,\xi_2,\ldots\},\qquad \xi_1=u.
\]
For $n\geq0$, define
\begin{equation}\label{eq:V-general}
 V_n=\Span_{\F_p}\left\{
  \prod_{\ell=1}^r\xi_{i_\ell}^{p^{e_\ell}}:
  r\geq0,\ e_\ell\geq0,\ \sum_{\ell=1}^ri_\ell\leq n
 \right\}.
\end{equation}
Exactly as in \cref{prop:V-properties}, these spaces form an exhaustive
increasing multiplicative filtration with
\begin{equation}\label{eq:properties-general}
 V_0=\F_p,\qquad V_nV_m\subseteq V_{n+m},\qquad
 v^p\in V_n\quad(v\in V_n).
\end{equation}
Indeed, Frobenius raises each exponent $p^e$ to $p^{e+1}$ without changing
the sum of the indices.

The escape property also persists:
\begin{equation}\label{eq:outside-general}
  \F_p[u]\nsubseteq V_E\qquad(E\geq0).
\end{equation}
To see this, fix $E$ and set $F_L=\F_p(u^{p^L})$.  The polynomial
$T^{p^L}-t$ is Eisenstein at $t$ over $\F_p[t]$, so
$1,u,\ldots,u^{p^L-1}$ are linearly independent over $F_L$.  After factors
with Frobenius exponent at least $L$ are absorbed into $F_L$, a monomial in
$V_E$ has at most $E$ factors chosen from at most $EL$ possibilities.  Hence
\[
  \dim_{F_L}(F_LV_E)\leq\sum_{r=0}^E(EL)^r.
\]
For $E=0$ the conclusion is immediate from $V_0=\F_p$; for fixed $E>0$,
the displayed bound is eventually smaller than $p^L$, proving
\eqref{eq:outside-general}.  We also have
\[
  u\in V_1,\qquad u^j\in V_j,\qquad u^{p^n}\in V_1.
\]

Let $c$ be transcendental over $K$ and put
\begin{equation}\label{eq:R-general}
 A=\bigoplus_{n\geq0}c^nV_n,\qquad
 \fm=A_+=\bigoplus_{n\geq1}c^nV_n,\qquad
 R_p=A_{\fm}.
\end{equation}
The proof of \cref{lem:formal-expansion} uses only multiplicativity, so every
element of $R_p$ again has a formal expansion whose $c^n$-coefficient lies
in $V_n$.  If $M_p=\fm R_p$, then Frobenius gives
\begin{equation}\label{eq:p-control-general}
  z^p\in cR_p\qquad(z\in M_p),
\end{equation}
because, for $b=\sum_{n\geq1}c^nb_n\in\fm$,
\[
  \frac{b^p}{c}=\sum_{n\geq1}c^{pn-1}b_n^p\in A.
\]
Here $b_n^p\in V_n\subseteq V_{pn-1}$.

The divisibility argument in \cref{lem:two-estimates}(b) is unchanged after
replacing $2^r$ by $p^r$.  Explicitly, if the least term of a nonzero
$b\in\fm$ is $c^\nu b_\nu$, choose $s$ with $b_\nu^{-1}\in V_s$ and $r$
so large that $(j-\nu)p^r\geq j+s$ for every higher nonzero term.  Then
\[
 \lambda:=c^sb_\nu^{-p^r}\in A,
 \qquad
 \lambda b^{p^r}=c^{\nu p^r+s}(1+w),\quad w\in\fm.
\]
Thus every $0\neq a\in M_p$ divides some power $c^D$.  Equations
\eqref{eq:p-control-general} and $c^D\in aR_p$ show, exactly as in
\cref{thm:R-char2}, that every nonzero proper ideal $I$ has SFT data
$(I,(a),pD)$ and that
\[
  \Spec R_p=\{(0),M_p\}.
\]
Hence $R_p$ is a one-dimensional local SFT domain.

Exhaustivity gives $\Frac(R_p)=K(c)$.  Let
\[
Q_p=\ker\bigl(R_p[X]\longrightarrow K(c),\ X\longmapsto u\bigr).
\]
The quotient is the domain $R_p[u]$, the map is injective on $R_p$, and
$c(X-u)\neq0$ lies in the kernel.  Thus $Q_p$ is an upper to zero.  It
contains
\begin{equation}\label{eq:gn-general}
  g_n=c(X-u)^{p^n}=cX^{p^n}-cu^{p^n}\qquad(n\geq0).
\end{equation}
Suppose that $Q_p$ had SFT data $(Q_p,J,e)$, where
$J=(h_1,\ldots,h_t)$.  Then $J\neq0$; discard zero generators and put
$d=\max_i\deg h_i\geq1$, since $Q_p\cap R_p=(0)$.  Choose $q=p^r\geq e$.
The division and coefficient arguments of
\cref{lem:division-char2,thm:Q-not-sft} are characteristic-free once
$u^j\in V_j$ is known.  They
give, after cancelling one factor $X-u$, identities
\[
  c^q(X-u)^{p^{n+r}-1}
    =\sum_{i=1}^t a_{i,n}\widetilde h_i,
\]
in which every $c^\ell$-coefficient of $\widetilde h_i$ lies in
$V_{\ell+d}$.  The finite convolution for the coefficient of $c^q$ then
puts every coefficient of $(X-u)^{p^{n+r}-1}$ in the fixed space $V_{q+d}$.
But
\[
  (X-u)^{p^N-1}
   =\frac{X^{p^N}-u^{p^N}}{X-u}
   =\sum_{j=0}^{p^N-1}u^jX^{p^N-1-j}.
\]
Letting $n$ grow forces $\F_p[u]\subseteq V_{q+d}$, contradicting
\eqref{eq:outside-general}.  Thus $Q_p$, and therefore $R_p[X]$, is not
SFT.
\end{proof}

\section{Characteristic zero}
\label{sec:characteristic-zero}

We first give a direct equicharacteristic-zero construction.  It replaces
Frobenius by a weighted filtration and a roots-of-unity norm.

\begin{theorem}\label{thm:equicharacteristic-zero}
There is a one-dimensional local SFT domain $R_0$ containing $\mathbb Q$
such that $R_0[X]$ is not SFT.  More precisely, $R_0[X]$ has a non-SFT
upper to zero.
\end{theorem}

\begin{proof}
Let $k=\overline{\mathbb Q}$, a countable field containing every root of
unity.  Let $u$ be transcendental over $k$ and put $K=k(u)$.  The field $K$
is countable, so enumerate
$K^\times=\{\xi_1,\xi_2,\ldots\}$ with $\xi_1=1$.
Use the labelled alphabet
\[
  \{a_j:j\geq0\}\amalg\{x_i:i\geq1\},
  \qquad a_j\mapsto u^{2^j},\quad x_i\mapsto\xi_i,
\]
with weights $w(a_j)=1$ and $w(x_i)=i$.  Labels are retained even when
their images in $K$ agree.  For $n\geq0$, define
\begin{equation}\label{eq:V-char0}
 V_n:=\Span_k\left\{
  \text{images of words of total weight at most }2^n-1
 \right\}\subseteq K,
\end{equation}
where the empty word is allowed.

These spaces are increasing and exhaustive, with $V_0=k$.  Concatenation
of words and the elementary budget inequalities give
\begin{equation}\label{eq:filtration-char0}
 V_nV_m\subseteq V_{n+m},\qquad
 V_iV_jV_\ell\subseteq V_{i+j+\ell-1}\quad(i,j,\ell\geq1).
\end{equation}
Indeed, if $B_r=2^r-1$, then
$B_n+B_m\leq B_{n+m}$ and
$B_i+B_j+B_\ell\leq B_{i+j+\ell-1}$ for positive indices.  We also have
\begin{equation}\label{eq:powers-char0}
 u^{2^j}\in V_1\quad(j\geq0),\qquad u^t\in V_t\quad(t\geq0).
\end{equation}

We need the escape property
\begin{equation}\label{eq:escape-char0}
  k[u]\nsubseteq V_E\qquad(E\geq0).
\end{equation}
Fix $E$ and set $B=2^E-1$.  A word of weight at most $B$ has image
$h(u)u^s$, where $h$ belongs to a finite set of products of
$\xi_1,\ldots,\xi_B$, and $s$ is a sum of at most $B$ powers of $2$.
Let $S_B$ be the set of all such $s$.  Then
\[
  \#(S_B\cap[0,L])=O_B((\log L)^B).
\]
After clearing one common denominator for the finite set of possible $h$,
there are a nonzero $\Delta\in k[u]$ and a finite set $F\subseteq\mathbb N$
such that
\[
  \operatorname{supp}(\Delta f)\subseteq F+S_B\qquad(f\in V_E).
\]
If $k[u]\subseteq V_E$, then, for any
$r\in\operatorname{supp}(\Delta)$, this would force
$r+\mathbb N\subseteq F+S_B$.  The left side has linear growth, whereas
the right side has $O_B((\log L)^B)$ elements up to $L$.  This proves
\eqref{eq:escape-char0}.

Let $c$ be transcendental over $K$ and set
\[
 A_0=\bigoplus_{n\geq0}c^nV_n\subseteq K[c],\qquad
 \fm_0=(A_0)_+,
 \qquad R_0=(A_0)_{\fm_0},\qquad M_0=\fm_0R_0.
\]
By \eqref{eq:filtration-char0}, $A_0$ is a graded domain and
$A_0/\fm_0\cong k$, so $R_0$ is local with maximal ideal $M_0$.  The proof
of \cref{lem:formal-expansion} uses only multiplicativity and therefore
shows that every $r\in R_0$ has a unique expansion
\begin{equation}\label{eq:formal-char0}
  r=\sum_{n\geq0}c^nr_n\in K[[c]],\qquad r_n\in V_n.
\end{equation}
The second inclusion in \eqref{eq:filtration-char0} gives
\begin{equation}\label{eq:cube-char0}
  z^3\in cR_0\qquad(z\in M_0),
\end{equation}
because every coefficient of $b^3/c$, for $b\in\fm_0$, is a sum of
triple products whose filtration indices drop by one; localization gives
the assertion for $M_0$.

We next show that every $0\neq a\in M_0$ divides a power of $c$.  Write
$a=b/t$ with $b\in\fm_0$ and $t\in A_0\setminus\fm_0$.  Since $t$ is a
unit, $aR_0=bR_0$, so it suffices to work with
\[
  b=\sum_{j=\nu}^{L}c^jb_j\in\fm_0,
  \qquad b_j\in V_j,\quad b_\nu\neq0.
\]
Choose $i_0$ with $\xi_{i_0}=b_\nu^{-1}$ and put $w_0=i_0$.
Choose $q$ so large that
\[
  q(2^L-1+w_0)\leq2^q-1.
\]
Since $T^q-1$ has exactly $q$ distinct roots in $k$, put
$\mu_q=\{\zeta\in k:\zeta^q=1\}$ and form
\[
  \mathcal N_q(b):=\prod_{\zeta\in\mu_q}b(\zeta c).
\]
After factoring the lowest term, invariance under $c\mapsto\eta c$ for
$\eta\in\mu_q$ gives
\begin{equation}\label{eq:norm-char0}
 \mathcal N_q(b)=\varepsilon c^{\nu q}b_\nu^q
  \left(1+\sum_{r\geq1}\theta_r c^{rq}\right),
  \qquad \varepsilon\in k^\times.
\end{equation}
Only finitely many $\theta_r$ are nonzero.
Each ratio $b_j/b_\nu$ is a $k$-linear combination of words of weight at
most $2^L-1+w_0$.  Hence every product of at most $q$ such ratios has weight
at most $q(2^L-1+w_0)\leq2^q-1$.  Thus
$\theta_r\in V_q\subseteq V_{rq}$.  The same bound applied to the $q$-fold
product of the label for $b_\nu^{-1}$ gives $b_\nu^{-q}\in V_q$.
Therefore
\[
 \lambda:=c^qb_\nu^{-q}\in A_0,
 \qquad \omega:=\sum_{r\geq1}\theta_rc^{rq}\in\fm_0.
\]
Every $b(\zeta c)$ lies in $A_0$, since $\zeta\in k$, and the factor for
$\zeta=1$ is $b$.  Thus $\mathcal N_q(b)\in bA_0$, and
\eqref{eq:norm-char0} yields
\[
 c^{\nu q+q}=\varepsilon^{-1}(1+\omega)^{-1}
   \lambda\mathcal N_q(b)\in bR_0=aR_0.
\]

Let $I$ be a nonzero proper ideal of $R_0$ and choose $0\neq a\in I$.
Since $I\subseteq M_0$, for some $D$ the preceding paragraph gives
$c^D\in aR_0$.  By
\eqref{eq:cube-char0}, every $z\in I$ satisfies
$z^{3D}\in c^DR_0\subseteq aR_0$.  Thus $I$ is SFT with data
$(I,(a),3D)$, and $R_0$ is SFT.  The same two facts show that every nonzero
prime contains $c$ and then $M_0$; hence
\[
  \Spec R_0=\{(0),M_0\},\qquad \dim R_0=1.
\]
In particular, $R_0$ is equicharacteristic zero, with residue field $k$.

Exhaustivity gives $\Frac(R_0)=K(c)$.  Define
\[
 Q_0:=\ker\bigl(R_0[X]\longrightarrow K(c),\ X\longmapsto u\bigr).
\]
This is a prime with $Q_0\cap R_0=(0)$, and it is nonzero because, for
$m=2^n$,
\begin{equation}\label{eq:gn-char0}
  g_n:=c(X^m-u^m)\in Q_0;
\end{equation}
indeed, $u^m\in V_1$.  Thus $Q_0$ is an upper to zero.

Suppose that $Q_0$ had SFT data $(Q_0,J,e)$, with
$J=(h_1,\ldots,h_t)$.  The ideal $J$ is nonzero because $R_0[X]$ is a
domain and $g_0\neq0$.  Discard zero generators, put
$d=\max_i\deg h_i\geq1$, and fix an integer $N\geq e$.  The proof of
\cref{lem:division-char2}, using \eqref{eq:powers-char0}, writes
$h_i=(X-u)\widetilde h_i$ so that every $c^\ell$-coefficient of every
$X$-coefficient of $\widetilde h_i$ lies in $V_{\ell+d-1}$.
Since $g_n^N\in J$, cancellation of $X-u$ gives
\begin{equation}\label{eq:cancelled-char0}
 c^N\frac{(X^m-u^m)^N}{X-u}
   =\sum_{i=1}^ta_{i,n}\widetilde h_i,
  \qquad a_{i,n}\in R_0[X].
\end{equation}
Use \eqref{eq:formal-char0} and compare the coefficient of $c^N$.
The convolution is finite, and a term with $c$-degrees $s$ and $\ell$ on
the right belongs to $V_sV_{\ell+d-1}\subseteq V_{N+d}$.  Consequently
every coefficient of
\[
  G_{m,N}:=\frac{(X^m-u^m)^N}{X-u}
\]
lies in the fixed space $V_{N+d}$.

For $0\leq t<mN$, set $s=\lfloor t/m\rfloor+1$.  Direct division gives
\begin{equation}\label{eq:coefficients-char0}
 [X^t]G_{m,N}
  =(-1)^{N-s}\binom{N-1}{s-1}u^{mN-1-t}.
\end{equation}
The scalar is nonzero because $k$ has characteristic zero.  Hence
$1,u,\ldots,u^{N2^n-1}\in V_{N+d}$ for every $n$.  Letting $n$ grow
contradicts \eqref{eq:escape-char0}.  Thus $Q_0$ is not SFT, completing
the proof.
\end{proof}

For comparison, the prime-characteristic examples also lift to characteristic
zero by a residue-field pullback.

\begin{theorem}\label{thm:mixed-characteristic}
For every prime $p$, there is a two-dimensional local SFT domain $T_p$ of
mixed characteristic $(0,p)$ such that $T_p[X]$ is not SFT.
\end{theorem}

\begin{proof}
Fix a prime $p$, and retain the objects
$R_p\subseteq K(c)=\F_p(u,c)$, $M_p$, and $Q_p$ from
\cref{sec:general-p}.  Localize at the height-one prime $(p)$ to obtain the
discrete valuation ring
\[
  W_p:=\mathbb Z[U,C]_{(p)}.
\]
Its maximal ideal is $pW_p$, its fraction field is $\mathbb Q(U,C)$, and its
residue field is $\F_p(U,C)$.  Identify this residue field with $K(c)$ by
$\overline U\mapsto u$ and $\overline C\mapsto c$, and let
\begin{equation}\label{eq:mixed-pullback}
  \rho_p:W_p\longrightarrow K(c),\qquad
  T_p:=\rho_p^{-1}(R_p),\qquad N_p:=\rho_p^{-1}(M_p).
\end{equation}
Then $T_p$ is a local domain with maximal ideal $N_p$: an element of $T_p$
is a unit precisely when its image in the local ring $R_p$ is a unit.
Moreover,
\[
  T_p/pW_p\cong R_p,\qquad T_p/N_p\cong\F_p.
\]
Thus $T_p$ has characteristic zero and residue characteristic $p$.

We determine its spectrum.  Since $pw\in T_p$ for every
$w\in W_p$,
\[
  T_p[1/p]=W_p[1/p]=\mathbb Q(U,C).
\]
Hence $(0)$ is the only prime of $T_p$ not containing $p$.  If a prime $P$
contains $p$ and $z=pw\in pW_p$, then
\[
  z^p=p\bigl(p^{p-1}w^p\bigr)\in P,
\]
where $p^{p-1}w^p\in pW_p\subseteq T_p$; hence $z\in P$.  Thus
$pW_p\subseteq P$, and the correspondence for $T_p/pW_p\cong R_p$,
together with \cref{thm:general-p}, gives
\[
  \Spec T_p=\{(0),pW_p,N_p\}.
\]
The inclusions are strict because $p\neq0$ and
$N_p/pW_p\cong M_p\neq(0)$.  In particular, $\dim T_p=2$.

We next verify the SFT property directly.  Let $I$ be a nonzero proper ideal
of $T_p$, and let $v_p$ denote the normalized valuation of the discrete
valuation ring $W_p$, with $v_p(0)=\infty$.  If $I\subseteq pW_p$, choose
$a\in I$ of minimal positive $p$-adic valuation.  For every $x\in I$,
\[
  v_p(x^p/a)\geq(p-1)v_p(a)\geq1,
\]
so $x^p\in aT_p$.  Thus $(I,(a),p)$ is SFT data.

Suppose instead that $I\nsubseteq pW_p$.  The nonzero ideal $\rho_p(I)$ of
the SFT ring $R_p$ has SFT data
\[
  \bigl(\rho_p(I),(\beta_1,\ldots,\beta_t),e\bigr).
\]
Choose lifts $b_i\in I$ of the $\beta_i$.  The ideal $I\cap pW_p$ is nonzero
(if $x\in I\setminus pW_p$, then $0\neq px\in I\cap pW_p$), so choose
$a\in I\cap pW_p$ of minimal $p$-adic valuation.  For $x\in I$, lift an
expression for $\rho_p(x)^e$ in $(\beta_1,\ldots,\beta_t)$ to obtain
\[
  x^e=y+z,\qquad y\in(b_1,\ldots,b_t)T_p,\quad
  z\in I\cap pW_p.
\]
Minimality gives $z^p\in aT_p$.  Every other term in the binomial expansion
\[
  x^{pe}=(y+z)^p
    =z^p+\sum_{k=1}^p\binom{p}{k}y^kz^{p-k}
\]
belongs to $(b_1,\ldots,b_t)T_p$.  Therefore
$x^{pe}\in(b_1,\ldots,b_t,a)T_p$.  The displayed finitely generated ideal
is contained in $I$, so $I$ is SFT.  The zero and unit ideals are immediate;
hence $T_p$ is an SFT ring.

Finally, reduction modulo $pW_p$ gives
\[
  T_p[X]/(pW_p)T_p[X]\cong R_p[X].
\]
Since the SFT property passes to quotients and $R_p[X]$ is not SFT by
\cref{thm:general-p}, the ring $T_p[X]$ is not SFT.  More precisely, the
inverse image $\widetilde Q_p$ of $Q_p$ is a non-SFT prime of $T_p[X]$: if
it were SFT, its image $Q_p$ would be SFT.  Its contraction is
\[
  \widetilde Q_p\cap T_p=pW_p,
\]
so this lifted prime is not an upper to zero.
\end{proof}

Thus the first construction is one-dimensional and equicharacteristic zero,
with a non-SFT upper to zero, whereas the pullback construction is
two-dimensional and mixed characteristic.

\section{Relation to earlier results}
\label{sec:discussion}

\paragraph*{Extended ideals.}
For the equicharacteristic-zero and prime-characteristic uppers to zero,
and for the mixed-characteristic prime, respectively,
\[
\begin{aligned}
 (Q_0\cap R_0)R_0[X]&=(0)\neq Q_0,\\
 (Q_p\cap R_p)R_p[X]&=(0)\neq Q_p,\\
 (\widetilde Q_p\cap T_p)T_p[X]&=(pW_p)T_p[X]\neq\widetilde Q_p.
\end{aligned}
\]
Thus none of the displayed non-SFT primes is extended from the coefficient
ring.  The theorem of Coykendall and Dutta that extensions of individual SFT
ideals remain SFT therefore does not apply to them.  This is why knowing that
every extended ideal $IR[X]$ is SFT does not settle whether $R[X]$ is SFT.

\paragraph*{Park's hypotheses.}
The examples lie outside the classes covered by Park's polynomial-extension
theorems.

\begin{proposition}\label{prop:not-normal}
The coefficient domains $R$ and $R_p$ from
\cref{sec:char2,sec:general-p}, and $R_0$ from
\cref{sec:characteristic-zero}, are not integrally closed.
\end{proposition}

\begin{proof}
We give the positive-characteristic arguments together.  Let $S=R$ with
$p=2$, or let $S=R_p$ for an arbitrary prime $p$.  Choose
$j\geq1$ such that $u^j\notin V_1$, which is possible by
\eqref{eq:outside-general} in the general construction and by
\cref{prop:V-properties}(d) in characteristic $2$.  Set
\[
  z=cu^j\in K(c)=\Frac(S).
\]
If $z$ belonged to $S$, its unique formal $c$-adic expansion would have
$c$-coefficient $u^j\in V_1$, a contradiction.  Thus $z\notin S$.

On the other hand, $u^j\in V_j$.  Choose $N$ with $p^N\geq j$.
Frobenius stability gives
\[
  u^{jp^N}=(u^j)^{p^N}\in V_j\subseteq V_{p^N},
\]
and hence
\[
  z^{p^N}=c^{p^N}u^{jp^N}\in c^{p^N}V_{p^N}\subseteq S.
\]
Therefore $z$ satisfies the monic polynomial
$T^{p^N}-z^{p^N}\in S[T]$ but does not belong to $S$.  Thus $z$ is
integral over $S$ and $S$ is not integrally closed.

For $R_0$, the definition \eqref{eq:V-char0} gives
\[
 V_1=\Span_k\bigl(\{1\}\cup\{u^{2^j}:j\geq0\}\bigr),
\]
so $u^3\notin V_1$.  Hence $cu^3\notin R_0$ by
\eqref{eq:formal-char0}.  On the other hand,
$u^6=(u^2)^3\in V_1^3\subseteq V_2$, and therefore
$(cu^3)^2=c^2u^6\in A_0\subseteq R_0$.  Thus $cu^3$ is integral over
$R_0$, proving the remaining case.
\end{proof}

The direct equicharacteristic domains are one-dimensional and not integrally
closed, hence neither zero-dimensional nor Pr\"ufer.  Since $(0)$ is prime,
they also fail the hypothesis that $R/P$ be integrally closed for every
$P\in\Spec R$.  The ring $T_p$ fails the same hypothesis at $P=pW_p$, because
$T_p/pW_p\cong R_p$ is not integrally closed.  Thus all the examples are
compatible with the positive results in \cite{Park2019}.

\paragraph*{Immediate consequences.}
\begin{corollary}\label{cor:basic-consequences}
Let $S$ be any of the counterexample coefficient domains
$R$, $R_p$, $R_0$, or $T_p$ constructed above.  Then:
\begin{enumerate}[label=\textup{(\alph*)}]
\item the displayed non-SFT prime of $S[X]$ is not finitely generated;
\item $S$ is non-Noetherian;
\item $S[X_1,\ldots,X_m]$ is not SFT for every $m\geq1$.
\end{enumerate}
\end{corollary}

\begin{proof}
A finitely generated ideal is SFT with exponent $1$, proving (a).  If $S$
were Noetherian, the Hilbert basis theorem would make $S[X]$ Noetherian and
hence SFT, proving (b).  The case $m=1$ of (c) is part of the construction.
If $m\geq2$ and $S[X_1,\ldots,X_m]$ were SFT, its quotient by
$(X_2,\ldots,X_m)$ would be SFT because the SFT property passes to
homomorphic images.  The quotient is $S[X_1]$, a contradiction.
\end{proof}

\paragraph*{Polynomial versus power series.}
The argument is specific to polynomial rings: evaluation at $X=u$ need not
extend to the corresponding formal-power-series rings.  We therefore make no
claim about their SFT property or dimension.

\section*{Acknowledgments}
We used GPT-5.6 Sol, Claude Fable 5, and an agentic harness built around these
models to assist with literature search, hypothesis testing, the exploration
and elimination of potential approaches, wording refinement, and manuscript
proofreading.  We are very grateful to Dung V. Nguyen and Quang X. Nguyen for
technical support in the use of the AI tools and agentic system.  Viet-Hoang
Tran thanks Tho Tran Huu for support with hardware-related matters.  The
authors independently verified and streamlined the proofs and were responsible
for composing the manuscript.

\bibliographystyle{plain}
\begingroup
\footnotesize
\raggedright
\bibliography{example_paper}
\endgroup

\end{document}